\documentclass[amstex,12pt,russian,amssymb]{article}

\usepackage{mathtext}
\usepackage[cp1251]{inputenc}
\usepackage[T2A]{fontenc}
\usepackage[russian]{babel}
\usepackage[dvips]{graphicx}
\usepackage{amsmath}
\usepackage{amssymb}
\usepackage{amsxtra}
\usepackage{latexsym}
\usepackage{ifthen}

\begin{document}

\newcounter{lemma}
\newcommand{\lemma}{\par \refstepcounter{lemma}%
{\bf Лемма \arabic{lemma}.}}

\newcounter{corollary}
\newcommand{\corollary}{\par \refstepcounter{corollary}%
{\bf Следствие \arabic{corollary}.}}

\newcounter{remark}
\newcommand{\remark}{\par \refstepcounter{remark}%
{\bf Замечание \arabic{remark}.}}

\newcounter{theorem}
\newcommand{\theorem}{\par \refstepcounter{theorem}%
{\bf Теорема \arabic{theorem}.}}

\newcounter{proposition}
\newcommand{\proposition}{\par \refstepcounter{proposition}%
{\bf Предложение \arabic{proposition}.}}

\renewcommand{\refname}{\centerline{\bf Список литературы}}

\newcommand{\proof}{{\it Доказательство.\,\,}}

\noindent УДК 517.5

{\bf Е.А.~Севостьянов} (Житомирский государственный университет им.\
И.~Франко)

\medskip
{\bf Є.О.~Севостьянов} (Житомирський державний університет ім.\
І.~Франко)

\medskip
{\bf E.A.~Sevost'yanov} (Zhitomir Ivan Franko State University)

\medskip
{\bf О граничном поведении отображений класса Соболева и
Орлича--Соболева}

{\bf Про граничну поведінку відображень класу Соболєва і
Орліча--Соболєва}

{\bf On boundary behavior of mappings of Sobolev and Orlicz--Sobolev
class}

\medskip
Изучается граничное поведение замкнутых открытых дискретных
отображений классов Соболева и Орлича--Соболева в ${\Bbb R}^n,$
$n\ge 2.$ Установлено, что указанные отображения $f$ имеют
непрерывное продолжение в граничную точку $x_0\in \partial D$
области $D\subset {\Bbb R}^n,$ как только их внутренняя дилатация
порядка $p\geqslant 1$ имеет мажоранту класса $FMO$ (конечного
среднего колебания) в указанной точке. Другим достаточным условием
возможности непрерывного продолжения указанных отображений является
расходимость некоторого интеграла.

\medskip
Вивчається гранична поведінка замкнених відкритих дискретних
відображень класів Соболєва і Орліча--Соболєва в ${\Bbb R}^n,$ $n\ge
2.$ Встановлено, що вказані відображення $f$ мають неперервне
продовження в межову точку $x_0\in \partial D$ області $D\subset
{\Bbb R}^n,$ як тільки їх внутрішня дилатація порядку $p\geqslant 1$
має мажоранту класу $FMO$ (скінченного середнього коливання) в
зазначеній точці. Іншою достатньою умовою неперервного продовження
вказаних відображень є розбіжність певного інтеграла.

\medskip
A boundary behavior of closed open discrete mappings of Sobolev and
Orlicz--Sobolev classes in ${\Bbb R}^n,$ $n\ge 3,$ is studied. It is
proved that, mappings mentioned above have a continuous extension to
boundary point $x_0$ of a domain $D$ whenever its inner dilatation
of order $p$ has a majorant $FMO$ (finite mean oscillation) at the
point. Another sufficient condition of possibility of continuous
extension is a divergence of some integral.

\newpage

{\bf 1. Введение.} В сравнительно недавних работах \cite{KRSS} и
\cite{KR$_1$} получены некоторые важные результаты о локальном и
граничном поведении классов Орлича--Соболева в $n$-мерном евклидовом
пространстве. В частности, в \cite[лемма~5 и теорема~15]{KRSS}
утверждается возможность непрерывного продолжения на границу
гомеоморфизмов указанных классов, как только их внешняя дилатация в
степени $n-1$ удовлетворяет некоторым ограничениям интегрального
характера. Ввиду \cite[теорема~2.1]{KR$_1$} и \cite[лемма~9.4]{MRSY}
эти утверждения могут быть несколько усилены: вместо внешней
дилатации может быть взята внутренняя дилатация в первой степени.
Основная цель настоящей заметки -- показать, что указанные
результаты справедливы не только для гомеоморфизмов, но и для более
широких классов открытых замкнутых дискретных отображений. Перейдём
к определениям.

\medskip
Всюду далее $D$ -- область в ${\Bbb R}^n,$ $n\ge 2,$ $m$ -- мера
Лебега в ${\Bbb R}^n.$ Здесь и далее {\it предельным множеством
отображения $f$ относительно множества $E\subset \overline{{\Bbb
R}^n}$} называется множество
$C(f, E):=\left\{y\in {\Bbb R}^n: \,\exists\,x_0\in E:
y=\lim\limits_{m\rightarrow \infty} f(x_m), x_m\rightarrow
x_0\right\}.$ Отображение $f:D\rightarrow {\Bbb R}^n$ называется
{\it дискретным}, если прообраз $f^{-1}\left(y\right)$ каждой точки
$y\in{\Bbb R}^n$ состоит только из изолированных точек. Отображение
$f:D\rightarrow {\Bbb R}^n$ называется {\it открытым}, если образ
любого открытого множества $U\subset D$ является открытым множеством
в ${\Bbb R}^n.$ Отображение $f:D\rightarrow {\Bbb R}^n$ называется
{\it сохраняющим границу отображением} (см. \cite[разд. 3, гл.
II]{Vu$_1$}), если выполнено соотношение $C(f,
\partial D)\subset \partial f(D).$ Отметим, что условие сохранения границы для
открытых дискретных отображений эквивалентно тому, что отображение
$f$ {\it замкнуто} (т.е., $f(A)$ замкнуто в $f(D)$ для любого
замкнутого $A\subset D$), а также тому, что $f^{\,-1}(K)$ компактно
в $D$ для любого компакта $K\subset f(D)$ (см.
\cite[теорема~3.3]{Vu$_1$}). Отображение $f:D\rightarrow {\Bbb R}^n$
называется {\it отображением с конечным искажением}, пишем $f\in
FD,$ если $f\in W_{loc}^{1,1}(D)$ и для некоторой функции $K(x):
D\rightarrow [1,\infty)$ выполнено условие
$\Vert f^{\,\prime}\left(x\right) \Vert^{n}\le K(x)\cdot |J(x,f)|$
при почти всех $x\in D,$ где $\Vert
f^{\,\prime}(x)\Vert=\max\limits_{h\in {\Bbb R}^n \setminus \{0\}}
\frac {|f^{\,\prime}(x)h|}{|h|}$ (см. \cite[п.~6.3, гл.~VI]{IM}.
Полагаем $l\left(f^{\,\prime}(x)\right)\,=\,\,\,\min\limits_{h\in
{\Bbb R}^n \setminus \{0\}} \frac {|f^{\,\prime}(x)h|}{|h|}.$
Отметим, что для отображений с конечным искажением и произвольного
$p\geqslant 1$ корректно определена и почти всюду конечна так
называемая {\it внутренняя дилатация $K_{I, p}(x,f)$ отображения $f$
порядка $p$ в точке $x$}, определяемая равенствами
\begin{equation}\label{eq0.1.1A}
K_{I, p}(x,f)\quad =\quad\left\{
\begin{array}{rr}
\frac{|J(x,f)|}{{l\left(f^{\,\prime}(x)\right)}^p}, & J(x,f)\ne 0,\\
1,  &  f^{\,\prime}(x)=0, \\
\infty, & \text{в\,\,остальных\,\,случаях}
\end{array}
\right.\,.
\end{equation}
Пусть $\varphi:[0,\infty)\rightarrow[0,\infty)$ -- неубывающая
функция, $f$ -- локально интегрируемая вектор-функция $n$
вещественных переменных $x_1,\ldots,x_n,$ $f=(f_1,\ldots,f_n),$
$f_i\in W_{loc}^{1,1},$ $i=1,\ldots,n.$ Будем говорить, что
$f:D\rightarrow {\Bbb R}^n$ принадлежит классу
$W^{1,\varphi}_{loc},$ пишем $f\in W^{1,\varphi}_{loc},$ если
$\int\limits_{G}\varphi\left(|\nabla f(x)|\right)\,dm(x)<\infty$ для
любой компактной подобласти $G\subset D,$ где $|\nabla
f(x)|=\sqrt{\sum\limits_{i=1}^n\sum\limits_{j=1}^n\left(\frac{\partial
f_i}{\partial x_j}\right)^2}.$ Класс $W^{1,\varphi}_{loc}$
называется классом {\it Орлича--Соболева}. Согласно \cite[c.~232,
п.~I, $\S\,49,$ гл.~6]{Ku} область $D$ называется {\it локально
связной в точ\-ке} $x_0\in\partial D,$ если для любой окрестности
$U$ точки $x_0$ найдется окрестность $V\subset U$ точки $x_0$ такая,
что $V\cap D$ связно. Происхождение следующего термина играет важную
роль при исследовании граничного поведения пространственных
отображений, см., напр., \cite[разд.~3.8]{MRSY}. Будем говорить, что
граница $\partial D$ области $D$ {\it сильно достижима в точке
$x_0\in
\partial D$ относительно $p$-модуля}, если для любой окрестности $U$ точки $x_0$ найдется
компакт $E\subset D,$ окрестность $V\subset U$ точки $x_0$ и число
$\delta
>0$ такие, что
\begin{equation}\label{eq17***}
M_p(\Gamma(E,F, D))\ge \delta
\end{equation}
для любого континуума  $F$ в $D,$ пересекающего $\partial U$ и
$\partial V.$ (Здесь $M_p$ обозначает модуль семейств кривых, а
$\Gamma(E,F, D)$ обозначает семейство всех кривых, соединяющих
множества $E$ и $F$ в области $D,$ см., напр., \cite[разделы~2.2 и
2.5]{MRSY}). Граница области $D\subset{{\Bbb R}^n}$ называется {\it
сильно достижимой относительно $p$-модуля}, если указанное выше
свойство выполнено в каждой точке $x_0\in\partial D.$

\medskip
Справедливо следующее утверждение.

\medskip
\begin{theorem}\label{th1}
{\sl\, Пусть область $D$ локально связна в каждой точке своей
границы, $n\ge 3,$ отображение $f:D\rightarrow {\Bbb R}^n$ класса
$W_{loc}^{1, \varphi}(D)$ с конечным искажением является
ограниченным, открытым, дискретным и замкнутым, а граница области
$D^{\,\prime}=f(D)$ является сильно достижимой относительно
$\alpha$-модуля, $n-1<\alpha\leqslant n.$ Тогда $f$ имеет
непрерывное продолжение в точку $x_0\in \partial D,$ если
\begin{equation}\label{eqOS3.0a}
\int\limits_{1}^{\infty}\left(\frac{t}{\varphi(t)}\right)^
{\frac{1}{n-2}}dt<\infty
\end{equation}
и, кроме того, найдётся измеримая по Лебегу функция $Q:D\rightarrow
[0, \infty],$ такая что $K_{I,\alpha}(x, f)\leqslant Q(x)$ при почти
всех $x\in D,$ некотором $\varepsilon_0>0$ и всех $\varepsilon\in
(0, \varepsilon_0)$ выполнены следующие условия:
\begin{equation}\label{eq9}
\int\limits_{\varepsilon}^{\varepsilon_0}
\frac{dt}{t^{\frac{n-1}{\alpha-1}}q_{x_0}^{\,\frac{1}{\alpha-1}}(t)}<\infty\,,\qquad
\int\limits_{0}^{\varepsilon_0}
\frac{dt}{t^{\frac{n-1}{\alpha-1}}q_{x_0}^{\,\frac{1}{\alpha-1}}(t)}=\infty\,.
\end{equation}
Здесь
$q_{x_0}(r):=\frac{1}{\omega_{n-1}r^{n-1}}\int\limits_{|x-x_0|=r}Q(x)\,d{\mathcal
H}^{n-1}$ обозначает среднее интегральное значение функции $Q$ над
сферой $S(x_0, r).$ В частности, заключение теоремы \ref{th1}
является верным, если $q_{x_0}(r)=\,O\left({\left(
\log{\frac{1}{r}}\right)}^{n-1}\right)$ при $r\rightarrow 0.$}
\end{theorem}

\medskip
\begin{remark}\label{rem2}
Условие (\ref{eqOS3.0a}) принадлежит Кальдерону и использовалось им
для решения задач несколько иного плана (см. \cite{Cal}).
\end{remark}

\medskip
{\bf 2. Вспомогательные сведения, основные леммы и доказательство
теоремы \ref{th1}.}  Доказательство основного результата статьи
опирается на некоторый аппарат, суть которого излагается ниже (см.,
напр., \cite{MRSY}). Напомним некоторые определения, связанные с
понятием поверхности, интеграла по поверхности, а также модулей
семейств кривых и поверхностей.

\medskip Пусть $\omega$ -- открытое множество в $\overline{{\Bbb
R}^k}:={\Bbb R}^k\cup\{\infty\},$ $k=1,\ldots,n-1.$ Непрерывное
отображение $S:\omega\rightarrow{\Bbb R}^n$ будем называть {\it
$k$-мерной поверхностью} $S$ в ${\Bbb R}^n.$ Число прообразов
$N(y, S)={\rm card}\,S^{-1}(y)={\rm card}\,\{x\in\omega:S(x)=y\},\
y\in{\Bbb R}^n$ будем называть {\it функцией кратности} поверхности
$S.$ Другими словами, $N(y, S)$ -- кратность накрытия точки $y$
поверхностью $S.$ Пусть $\rho:{\Bbb R}^n\rightarrow\overline{{\Bbb
R}^+}$ -- борелевская функция, в таком случае интеграл от функции
$\rho$ по поверхности $S$ определяется равенством:  $\int\limits_S
\rho\,d{\mathcal{A}}:=\int\limits_{{\Bbb R}^n}\rho(y)\,N(y,
S)\,d{\mathcal H}^ky.$
Пусть $\Gamma$ -- семейство $k$-мерных поверхностей $S.$ Борелевскую
функцию $\rho:{\Bbb R}^n\rightarrow\overline{{\Bbb R}^+}$ будем
называть {\it допустимой} для семейства $\Gamma,$ сокр. $\rho\in{\rm
adm}\,\Gamma,$ если
\begin{equation}\label{eq8.2.6}\int\limits_S\rho^k\,d{\mathcal{A}}\geqslant 1\end{equation}
для каждой поверхности $S\in\Gamma.$ Пусть $p\geqslant 1,$ тогда
{\it $p$-модулем} семейства $\Gamma$ назовём величину
$$M_p(\Gamma)=\inf\limits_{\rho\in{\rm adm}\,\Gamma}
\int\limits_{{\Bbb R}^n}\rho^p(x)\,dm(x)\,.$$ Заметим, что
$p$-модуль семейств поверхностей, определённый таким образом,
представляет собой внешнюю меру в пространстве всех $k$-мер\-ных
поверхностей (см. \cite{Fu}). Говорят, что некоторое свойство $P$
выполнено для {\it $p$-почти всех поверхностей} области $D,$ если
оно имеет место для всех поверхностей, лежащих в $D,$ кроме, быть
может, некоторого их подсемейства, $p$-модуль которого равен нулю.
При $p=n$ приставка <<$p$->> в словах <<$p$-почти всех...>>, как
правило, опускается. В частности, говорят, что некоторое свойство
выполнено для {\it $p$-почти всех кривых} области $D$, если оно
имеет место для всех кривых, лежащих в $D$, кроме, быть может,
некоторого их подсемейства, $p$-модуль которого равен нулю.

Будем говорить, что измеримая по Лебегу функция $\rho:{\Bbb
R}^n\rightarrow\overline{{\Bbb R}^+}$ {\it $p$-обобщённо допустима}
для семейства $\Gamma$ $k$-мерных поверхностей $S$ в ${\Bbb R}^n,$
сокр. $\rho\in{\rm ext}_p\,{\rm adm}\,\Gamma,$ если соотношение
(\ref{eq8.2.6}) выполнено для $p$-почти всех поверхностей $S$
семейства $\Gamma.$ {\it Обобщённый $p$-модуль} $\overline
M_p(\Gamma)$ семейства $\Gamma$ определяется равенством
$$\overline{M_p}(\Gamma)= \inf\int\limits_{{\Bbb
R}^n}\rho^p(x)\,dm(x)\,,$$
где точная нижняя грань берётся по всем функциям $\rho\in{\rm
ext}_p\,{\rm adm}\,\Gamma.$ Очевидно, что при каждом
$p\in(0,\infty),$ $k=1,\ldots,n-1,$ и каждого семейства $k$-мерных
поверхностей $\Gamma$ в ${\Bbb R}^n,$ выполнено равенство
$\overline{M_p}(\Gamma)=M_p(\Gamma).$

Следующий класс отображений представляет собой обобщение
квазиконформных отображений в смысле кольцевого определения по
Герингу (\cite{Ge$_3$}) и отдельно исследуется (см., напр.,
\cite[глава~9]{MRSY}). Пусть $p\geqslant 1,$ $D$ и $D^{\,\prime}$ --
заданные области в $\overline{{\Bbb R}^n},$ $n\geqslant 2,$
$x_0\in\overline{D}\setminus\{\infty\}$ и $Q:D\rightarrow(0,\infty)$
-- измеримая по Лебегу функция. Будем говорить, что $f:D\rightarrow
D^{\,\prime}$ -- {\it нижнее $Q$-отображение в точке $x_0$
относительно $p$-модуля,} как только
\begin{equation}\label{eq1A}
M_p(f(\Sigma_{\varepsilon}))\geqslant \inf\limits_{\rho\in{\rm
ext}_p\,{\rm adm}\,\Sigma_{\varepsilon}}\int\limits_{D\cap
A(\varepsilon, r_0, x_0)}\frac{\rho^p(x)}{Q(x)}\,dm(x)
\end{equation}
для каждого кольца $A(x_0, \varepsilon, r_0)=\{x\in {\Bbb R}^n:
\varepsilon<|x-x_0|<r_0\},$ $r_0\in(0, d_0),$ $d_0=\sup\limits_{x\in
D}|x-x_0|,$
где $\Sigma_{\varepsilon}$ обозначает семейство всех пересечений
сфер $S(x_0, r)$ с областью $D,$ $r\in (\varepsilon, r_0).$ Примеры
таких отображений несложно указать (см. теорему \ref{thOS4.1} ниже).

Отметим, что выражения <<почти всех кривых>> и <<почти всех
по\-вер\-х\-но\-с\-тей>> в отдельных случаях могут иметь две
различные интерпретации (в частности, если речь идёт о семействе
сфер, то <<почти всех>> может пониматься как относительно множества
значений $r,$ так и $p$-модуля семейства сфер, рассматриваемого как
частный случай семейства поверхностей). Следующее утверждение вносит
некоторую ясность между указанными интерпретациями и может быть
установлено полностью по аналогии с \cite[лемма~9.1]{MRSY}.

\medskip
\begin{lemma}\label{lemma8.2.11}{\, Пусть $p\geqslant 1,$ $x_0\in D.$ Если некоторое
свойство $P$ имеет место для $p$-почти всех сфер $D(x_0, r):=S(x_0,
r)\cap D,$ где <<почти всех>> понимается в смысле модуля семейств
поверхностей, то $P$ также имеет место для почти всех сфер $D(x_0,
r)$ относительно линейной меры Лебега по параметру $r\in {\Bbb R }.$
Обратно, пусть $P$ имеет место для почти всех сфер $D(x_0,
r):=S(x_0, r)\cap D$ относительно линейной меры Лебега по $r\in
{\Bbb R},$ тогда $P$ также имеет место для $p$-почти всех
поверхностей $D(x_0, r):=S(x_0, r)\cap D$ в смысле модуля семейств
поверхностей.}\end{lemma}

\medskip
Следующее утверждение облегчает проверку бесконечной серии
неравенств в (\ref{eq1A}) и может быть установлено аналогично
доказательству \cite[теорема~9.2]{MRSY} (см. также
\cite[теорема~6.1]{GS}).

\medskip
\begin{lemma}\label{lemma4}{\,
Пусть $D,$  $D^{\,\prime}\subset\overline{{\Bbb R}^n},$
$x_0\in\overline{D}\setminus\{\infty\}$ и $Q:D\rightarrow(0,\infty)$
-- измеримая по Лебегу функция. Отображение $f:D\rightarrow
D^{\,\prime}$ является нижним $Q$-отображением относительно
$p$-модуля в точке $x_0,$ $p>n-1,$ тогда и только тогда, когда
%
$M_p(f(\Sigma_{\varepsilon}))\geqslant\int\limits_{\varepsilon}^{r_0}
\frac{dr}{||\,Q||_{s}(r)}\quad\forall\ \varepsilon\in(0,r_0)\,,\
r_0\in(0,d_0),$ $d_0=\sup\limits_{x\in D}|x-x_0|,$
%
$s=\frac{n-1}{p-n+1},$ где, как и выше, $\Sigma_{\varepsilon}$
обозначает семейство всех пересечений сфер $S(x_0, r)$ с областью
$D,$ $r\in (\varepsilon, r_0),$
$ \Vert
Q\Vert_{s}(r)=\left(\int\limits_{D(x_0,r)}Q^{s}(x)\,d{\mathcal{A}}\right)^{\frac{1}{s}}$
-- $L_{s}$-норма функции $Q$ над сферой $D(x_0,r)=\{x\in D:
|x-x_0|=r\}=D\cap S(x_0,r)$.}
\end{lemma}

\medskip
Пусть $G$ -- открытое множество в ${\Bbb R}^n$ и $I=\{x\in{\Bbb
R}^n:a_i<x_i<b_i,i=1,\ldots,n\}$ -- открытый $n$-мерный интервал.
Отображение $f:I\rightarrow{\Bbb R}^n$ {\it принадлежит классу
$ACL$} ({\it абсолютно непрерывно на линиях}), если $f$ абсолютно
непрерывно на почти всех линейных сегментах в $I,$ параллельных
координатным осям. Отображение $f:G\rightarrow{\Bbb R}^n$ {\it
принадлежит классу $ACL$} в $G,$ когда сужение $f|_I$ принадлежит
классу $ACL$ для каждого интервала $I,$ $\overline{I}\subset G.$

\medskip
Напомним, что {\it конденсатором} называют пару
$E=\left(A,\,C\right),$ где $A$ -- открытое множество в ${\Bbb
R}^n,$ а $C$ -- компактное подмножество $A.$ {\it Ёмкостью}
конденсатора $E$ порядка $p\geqslant 1$ называется следующая
величина:
%
%
${\rm cap}_p\,E={\rm cap}_p\,\left(A,\,C\right)= \inf\limits_{u\in
W_0(E)}\,\,\int\limits_A |\nabla u(x)|^p\,\,dm(x),$
%
где $W_0(E)=W_0\left(A,\,C\right)$ -- семейство неотрицательных
непрерывных функций $u:A\rightarrow{\Bbb R}$ с компактным носителем
в $A,$ таких что $u(x)\geqslant 1$ при $x\in C$ и $u\in ACL.$
Здесь, как обычно, $|\nabla
u|={\left(\sum\limits_{i=1}^n\,{\left(\partial_i u\right)}^2
\right)}^{1/2}.$

\medskip
Следующие важные сведения, касающиеся ёмкости пары множеств
относительно области, могут быть найдены в работе В.~Цимера
\cite{Zi}. Пусть $G$ -- ограниченная область в ${\Bbb R}^n$ и $C_{0}
, C_{1}$ -- непересекающиеся компактные множества, лежащие в
замыкании $G.$ Полагаем  $R=G \setminus (C_{0} \cup C_{1})$ и
$R^{\,*}=R \cup C_{0}\cup C_{1},$ тогда {\it $p$-ёмкостью пары
$C_{0}, C_{1}$ относительно замыкания $G$} называется величина
$C_p[G, C_{0}, C_{1}] = \inf \int\limits_{R} \vert \nabla u
\vert^{p}\ dm(x),$
где точная нижняя грань берётся по всем функциям $u,$ непрерывным в
$R^{\,*},$ $u\in ACL(R),$ таким что $u=1$ на $C_{1}$ и $u=0$ на
$C_{0}.$ Указанные функции будем называть {\it допустимыми} для
величины $C_p [G, C_{0}, C_{1}].$ Мы будем говорить, что  {\it
множество $\sigma \subset {\Bbb R}^n$ разделяет $C_{0}$ и $C_{1}$ в
$R^{\,*}$}, если $\sigma \cap R$ замкнуто в $R$ и найдутся
непересекающиеся множества $A$ и $B,$ являющиеся открытыми в
$R^{\,*} \setminus \sigma,$ такие что $R^{\,*} \setminus \sigma =
A\cup B,$ $C_{0}\subset A$ и $C_{1} \subset B.$ Пусть $\Sigma$
обозначает класс всех множеств, разделяющих $C_{0}$ и $C_{1}$ в
$R^{\,*}.$ Для числа $p^{\prime} = p/(p-1)$ определим величину
%
$$\widetilde{M}_{p^{\prime}}(\Sigma)=\inf\limits_{\rho\in
\widetilde{\rm adm} \Sigma} \int\limits_{{\Bbb
R}^n}\rho^{\,p^{\prime}}dm(x)\,,$$
%
где запись $\rho\in \widetilde{\rm adm}\,\Sigma$ означает, что
$\rho$ -- неотрицательная борелевская функция в ${\Bbb R}^n$ такая,
что
\begin{equation} \label{eq13.4.13}
\int\limits_{\sigma \cap R}\rho d{\mathcal H}^{n-1} \geqslant
1\quad\forall\, \sigma \in \Sigma\,.
\end{equation}
Заметим, что согласно результата Цимера
\begin{equation}\label{eq3}
\widetilde{M}_{p^{\,\prime}}(\Sigma)=C_p[G , C_{0} ,
C_{1}]^{\,-1/(p-1)}\,,
\end{equation}
см. \cite[теорема~3.13]{Zi} при $p=n$ и \cite[с.~50]{Zi$_1$} при
$1<p<\infty.$ Заметим также, что согласно результата Шлык
\begin{equation}\label{eq4}
M_p(\Gamma(E, F, D))= C_p[D, E, F]\,,
\end{equation}
см. \cite[теорема~1]{Shl}.

\medskip
Напомним, что отображение $f:X\rightarrow Y$ между пространствами с
мерами $(X, \Sigma, \mu)$ и $(Y, \Sigma^{\,\prime}, \mu^{\,\prime})$
обладает {\it $N$-свой\-с\-т\-вом} (Лузина), если из условия
$\mu(S)=0$ следует, что $\mu^{\,\prime}(f(S))=0.$ Следующее
вспомогательное утверждение получено в работе \cite{KRSS} (см.
теорема 1 и следствие 2).

\medskip
\begin{proposition}\label{pr1}
{\, Пусть $D$ -- область в ${\Bbb R}^n,$ $n\geqslant 3,$
$\varphi:(0,\infty)\rightarrow (0,\infty)$ -- неубывающая функция,
удовлетворяющая условию (\ref{eqOS3.0a}). Тогда:

1) Если $f:D\rightarrow{\Bbb R}^n$ -- непрерывное открытое
отображение класса $W^{1,\varphi}_{loc}(D),$ то $f$ имеет почти
всюду полный дифференциал в $D;$

2) Любое непрерывное отображение $f\in W^{1,\varphi}_{loc}$ обладает
$N$-свойством относительно $(n-1)$-мерной меры Хаусдорфа, более
того, локально абсолютно непрерывно на почти всех сферах $S(x_0, r)$
с центром в заданной предписанной точке $x_0\in{\Bbb R}^n$. Кроме
того, на почти всех таких сферах $S(x_0, r)$ выполнено условие
${\mathcal H}^{n-1}(f(E))=0,$ как только $|\nabla f|=0$ на множестве
$E\subset S(x_0, r).$ (Здесь <<почти всех>> понимается относительно
линейной меры Лебега по параметру $r$).}

\end{proposition}

\medskip
Для отображения $f:D\,\rightarrow\,{\Bbb R}^n,$ множества $E\subset
D$ и $y\,\in\,{\Bbb R}^n,$  определим {\it функцию кратности $N(y,
f, E)$} как число прообразов точки $y$ во множестве $E,$ т.е.
\begin{equation}\label{eq1.7A}
N(y, f, E)\,=\,{\rm card}\,\left\{x\in E: f(x)=y\right\}\,,\quad
%
N(f, E)\,=\,\sup\limits_{y\in{\Bbb R}^n}\,N(y, f, E)\,.
\end{equation}
Обозначим через $J_{n-1}f_r(a)$ величину, означающую $(n-1)$-мерный
якобиан сужения отображения $f$ на сферу $S(x_0, r)\supset a$ в
точке $a$ (см. \cite[раздел~3.2.1]{Fe}). Предположим, что
отображение $f:D\rightarrow {\Bbb R}^n$ дифференцируемо в точке
$x_0\in D$ и матрица Якоби $f^{\,\prime}(x_0)$ невырождена, $J(x_0,
f)={\rm det\,}f^{\,\prime}(x_0)\ne 0.$ Тогда найдутся системы
векторов $e_1,\ldots, e_n$ и
$\widetilde{e_1},\ldots,\widetilde{e_n}$ и положительные числа
$\lambda_1(x_0),\ldots,\lambda_n(x_0),$
$\lambda_1(x_0)\leqslant\ldots\leqslant\lambda_n(x_0),$ такие что
$f^{\,\prime}(x_0)e_i=\lambda_i(x_0)\widetilde{e_i}$ (см.
\cite[теорема~2.1 гл. I]{Re}), при этом,
\begin{equation}\label{eq11C}
|J(x_0, f)|=\lambda_1(x_0)\ldots\lambda_n(x_0),\quad \Vert
f^{\,\prime}(x_0)\Vert =\lambda_n(x_0)\,, \quad
l(f^{\,\prime}(x))=\lambda_1(x_0)\,,\end{equation}
\begin{equation}\label{eq41}
K_{I, p}(x_0,
f)=\frac{\lambda_1(x_0)\cdots\lambda_n(x_0)}{\lambda^p_1(x_0)}\,,
\end{equation}
см. \cite[соотношение~(2.5), разд.~2.1, гл.~I]{Re}. Числа
$\lambda_1(x_0),\ldots\lambda_n(x_0)$ называются {\it главными
значениями}, а вектора $e_1,\ldots, e_n$ и
$\widetilde{e_1},\ldots,\widetilde{e_n}$ -- {\it главными векторами
} отображения $f^{\,\prime}(x_0).$ Из геометрического смысла
$(n-1)$-мерного якобиана, а также первого соотношения в
(\ref{eq11C}) вытекает, что
\begin{equation}\label{eq10C}
\lambda_1(x_0)\cdots\lambda_{n-1}(x_0)\leqslant
J_{n-1}f_r(x_0)\leqslant \lambda_2(x_0)\cdots\lambda_n(x_0)\,,
\end{equation}
в частности, из (\ref{eq10C}) следует, что $J_{n-1}f_r(x_0)$
положителен во всех тех точках $x_0,$ где положителен якобиан
$J(x_0, f).$

\medskip
Следующие две леммы несут в себе основную смысловую нагрузку данной
заметки. Первое из них впервые установлено для случая гомеоморфизмов
в работе \cite{KR$_1$} (см. теорему 2.1).

\medskip
\begin{lemma}{}\label{thOS4.1} { Пусть $D$ -- область в ${\Bbb R}^n,$
$n\geqslant 2,$ $\varphi:(0,\infty)\rightarrow (0,\infty)$ --
неубывающая функция, удовлетворяющая условию (\ref{eqOS3.0a}).
Если $p>n-1,$ то каждое открытое дискретное отображение
$f:D\rightarrow {\Bbb R}^n$ с конечным искажением класса
$W^{1,\varphi}_{loc}$ такое, что $N(f, D)<\infty,$ является нижним
$Q$-отображением относительно $p$-модуля в каждой точке
$x_0\in\overline{D}$ при
$$Q(x)=N(f, D)\cdot K^{\frac{p-n+1}{n-1}}_{I, \alpha}(x, f),$$
$\alpha:=\frac{p}{p-n+1},$ где внутренняя дилатация $K_{I,\alpha}(x,
f)$ отображения $f$ в точке $x$ порядка $\alpha$ определена
соотношением (\ref{eq0.1.1A}), а кратность $N(f, D)$ определена
вторым соотношением в (\ref{eq1.7A}).}
\end{lemma}

\medskip
\begin{proof}
Заметим, что $f$ дифференцируемо почти всюду ввиду предложения
\ref{pr1}. Пусть $B$ -- борелево множество всех точек $x\in D,$ в
которых $f$ имеет полный дифференциал $f^{\,\prime}(x)$ и $J(x,
f)\ne 0.$ Применяя теорему Кирсбрауна и свойство единственности
аппроксимативного дифференциала (см. \cite[пункты~2.10.43 и
3.1.2]{Fe}), мы видим, что множество $B$ представляет собой не более
чем счётное объединение борелевских множеств $B_l,$
$l=1,2,\ldots\,,$ таких, что сужения $f_l=f|_{B_l}$ являются
билипшецевыми гомеоморфизмами (см., напр., \cite[пункты~3.2.2, 3.1.4
и 3.1.8]{Fe}). Без ограничения общности, мы можем полагать, что
множества $B_l$ попарно не пересекаются. Обозначим также символом
$B_*$ множество всех точек $x\in D,$ в которых $f$ имеет полный
дифференциал, однако, $f^{\,\prime}(x)=0.$

\medskip
Ввиду построения, множество $B_0:=D\setminus \left(B\bigcup
B_*\right)$ имеет лебегову меру нуль. Следовательно, по
\cite[теорема~9.1]{MRSY}, ${\mathcal H}^{n-1}(B_0\cap S_r)=0$ для
$p$-почти всех сфер $S_r:=S(x_0,r)$ с центром в точке
$x_0\in\overline{D},$ где <<$p$-почти всех>> следует понимать в
смысле $p$-модуля семейств поверхностей. По лемме \ref{lemma8.2.11}
также ${\mathcal H}^{n-1}(B_0\cap S_r)=0$ при почти всех $r\in {\Bbb
R}.$

\medskip
По предложению \ref{pr1} и из условия ${\mathcal H}^{n-1}(B_0\cap
S_r)=0$ для почти всех $r\in {\Bbb R}$ вытекает, что ${\mathcal
H}^{n-1}(f(B_0\cap S_r))=0$ для почти всех $r\in {\Bbb R}.$ По этому
предложению также ${\mathcal H}^{n-1}(f(B_*\cap S_r))=0,$ поскольку
$f$ -- отображение с конечным искажением и, значит, $\nabla f=0$
почти всюду, где $J(x, f)=0.$

\medskip
Пусть $\Gamma$ -- семейство всех пересечений сфер $S_r,$
$r\in(\varepsilon, r_0),$ $r_0<d_0=\sup\limits_{x\in D}\,|x-x_0|,$ с
областью $D$ (здесь $\varepsilon$ -- произвольное фиксированное
число из интервала $(0, r_0)$). Для заданной функции $\rho_*\in{\rm
adm}\,f(\Gamma),$ $\rho_*\equiv0$ вне $f(D),$ полагаем $\rho\equiv
0$ вне $B,$
$$\rho(x)\ \colon=\ \rho_*(f(x))\left(\frac{|J(x, f)|}{l(f^{\,\prime}(x))}
\right)^{\frac{1}{n-1}} \qquad\text{при}\ x\in B\,.$$
Учитывая соотношения (\ref{eq11C}) и (\ref{eq10C}),
\begin{equation}\label{eq12C}
\frac{|J(x, f)|}{l(f^{\,\prime}(x))} \geqslant J_{n-1}f_r(x)\,.
\end{equation}
Пусть $D_{r}^{\,*}\in f(\Gamma),$ $D_{r}^{\,*}=f(D\cap S_r).$
Заметим, что
$D_{r}^{\,*}=\bigcup\limits_{i=0}^{\infty} f(S_r\cap B_i)\bigcup
f(S_r\cap B_*)$
и, следовательно, для почти всех $r\in (\varepsilon, r_0)$
\begin{equation}\label{eq10B}
1\leqslant \int\limits_{D^{\,*}_r}\rho^{n-1}_*(y)d{\mathcal A_*}
\leqslant \sum\limits_{i=0}^{\infty} \int \limits_{f(S_r\cap B_i)}
\rho^{n-1}_*(y)N (y, f, S_r\cap B_i)d{\mathcal H}^{n-1}y +
\end{equation}
$$+\int\limits_{f(S_r\cap B_*)} \rho^{n-1}_*(y)N (y, f, S_r\cap B_*
)d{\mathcal H}^{n-1}y\,.$$ Учитывая доказанное выше, из
(\ref{eq10B}) мы получаем, что
\begin{equation}\label{eq11B}
1\leqslant \int\limits_{D^{\,*}_r}\rho^{n-1}_*(y)d{\mathcal A_*}
\leqslant \sum\limits_{i=1}^{\infty} \int \limits_{f(S_r\cap B_i)}
\rho^{n-1}_*(y)N (y, f, S_r\cap B_i)d{\mathcal H}^{n-1}y
\end{equation}
для почти всех $r\in (\varepsilon, r_0).$
Рассуждая покусочно на $B_i,$ $i=1,2,\ldots,$ ввиду \cite[1.7.6 и
теорема~3.2.5]{Fe} и (\ref{eq12C}) мы получаем, что
$$\int\limits_{B_i\cap S_r}\rho^{n-1}\,d{\mathcal A}=
\int\limits_{B_i\cap S_r}\rho_*^{n-1}(f(x))\frac{|J(x,
f)|}{l(f^{\,\prime}(x))}\,d{\mathcal A}=$$
$$=\int\limits_{B_i\cap S_r}\rho_*^{n-1}(f(x))\cdot \frac{|J(x,
f)|}{l(f^{\,\prime}(x))J_{n-1}f_r(x)}\cdot J_{n-1}f_r(x)\,d{\mathcal
A}\geqslant $$
\begin{equation}\label{eq12B}
\geqslant\int\limits_{B_i\cap S_r}\rho_*^{n-1}(f(x))\cdot
J_{n-1}f_r(x)\,d{\mathcal A}=\int\limits_{f(B_i\cap
S_r)}\rho_{*}^{n-1}\,N(y, f, S_r\cap B_i)d{\mathcal H}^{n-1}y
\end{equation} для почти всех $r\in (\varepsilon, r_0).$
Из (\ref{eq11B}) и (\ref{eq12B}) вытекает, что
$\rho\in{\rm{ext\,adm}}\,\Gamma.$

Замена переменных на каждом $B_l,$ $l=1,2,\ldots\,,$ (см., напр.,
\cite[теорема~3.2.5]{Fe}) и свойство счётной аддитивности интеграла
приводят к оценке
$$\int\limits_{D}\frac{\rho^p(x)}{K^{\frac{p-n+1}{n-1}}_{I,
\alpha}(x, f)}\,dm(x)\leqslant \int\limits_{f(D)}N(f, D)\cdot
\rho^{\,p}_*(y)\, dm(y)\,,$$ $\alpha:=\frac{p}{p-n+1},$ что и
завершает доказательство.
\end{proof}

\medskip
\begin{remark}\label{rem1}
Заключение леммы \ref{thOS4.1} при $n=2$ остаётся справедливым для
классов Соболева $W_{loc}^{1, 1}$ при аналогичных условиях, за
исключением дополнительного условия Кальдерона (\ref{eqOS3.0a}).
Чтобы в этом убедиться, необходимо повторить доказательство этой
леммы при $n=2,$ где необходимо учесть наличие $N$-свойства
указанных отображений на почти всех окружностях, что обеспечивается
свойством $ACL$ для произвольных классов Соболева (см.
\cite[теорема~1, п.~1.1.3, $\S$~1.1, гл.~I]{Maz}).
\end{remark}

\medskip
\begin{lemma}\label{lem1}
{\sl\, Пусть область $D$ локально связна в каждой точке своей
границы, $n\ge 2,$ отображение $f:D\rightarrow {\Bbb R}^n$ является
ограниченным, открытым, дискретным, и замкнутым нижним
$Q$-отображением относительно $p$-модуля, $n-1<p\leqslant n,$ а
граница области $D^{\,\prime}=f(D)$ является сильно достижимой
относительно $\alpha$-модуля, $\alpha:=\frac{p}{p-n+1}.$ Тогда
отображение $f:D\rightarrow {\Bbb R}^n$ имеет непрерывное
продолжение в точку $b\in\partial D,$ если при некотором
$\varepsilon_0>0$ и всех $\varepsilon\in (0, \varepsilon_0)$
выполнены условия
\begin{equation}\label{eq9A}
\int\limits_{\varepsilon}^{\varepsilon_0}
\frac{dt}{t^{\frac{n-1}{\alpha-1}}\widetilde{q}_{b}^{\,\frac{1}{\alpha-1}}(t)}<\infty\,,\qquad
\int\limits_{0}^{\varepsilon_0}
\frac{dt}{t^{\frac{n-1}{\alpha-1}}\widetilde{q}_{b}^{\,\frac{1}{\alpha-1}}(t)}=\infty\,,
\end{equation}
где $\alpha=\frac{p}{p-n+1},$
$\widetilde{q}_{b}(r):=\frac{1}{\omega_{n-1}r^{n-1}}\int\limits_{|x-b|=r}Q^{\frac{n-1}{p-n+1}}(x)\,d{\mathcal
H}^{n-1}$ обозначает среднее интегральное значение функции
$Q^{\frac{n-1}{p-n+1}}(x)$ над сферой $S(b, r).$}
\end{lemma}

\medskip
\begin{proof}
Предположим противное. Тогда найдутся, по крайней мере, две
последовательности $x_i,$ $x_i^{\,\prime}\in D,$ $i=1,2,\ldots,$
такие, что $x_i\rightarrow b,$ $x_i^{\,\prime}\rightarrow b$ при
$i\rightarrow \infty,$ $f(x_i)\rightarrow y,$
$f(x_i^{\,\prime})\rightarrow y^{\,\prime}$ при $i\rightarrow
\infty$ и $y^{\,\prime}\ne y.$ По определению сильно достижимой
границы в точке $y\in \partial D^{\,\prime}$ относительно
$\alpha$-модуля, для окрестности $U$ этой точки, не содержащей точки
$y^{\,\prime},$ найдутся компакт $C_0^{\,\prime}\subset
D^{\,\prime},$ окрестность $V$ точки $y,$ $V\subset U,$ и число
$\delta>0$ такие, что
\begin{equation}\label{eq1}
M_{\alpha}(\Gamma(C_0^{\,\prime}, F, D^{\,\prime}))\ge \delta
>0\,,\quad \alpha:=p/(p-n+1)\,.
\end{equation}
для произвольного континуума $F,$ пересекающего $\partial U$ и
$\partial V.$ Поскольку область $D$ локально связна в точке $b,$
можно соединить точки $x_i$ и $x_i^{\,\prime}$ кривой $\gamma_i,$
лежащей в достаточно малой окрестности $V_i\cap D$ точки $b.$ Можно
также считать, что $\gamma_i\subset \overline{B(b, 2^{\,-i})}\cap
D.$ Поскольку $f(x_i)\in V$ и $f(x_i^{\,\prime})\in D\setminus
\overline{U}$ при всех достаточно больших $i\in {\Bbb N},$ найдётся
номер $i_0\in {\Bbb N},$ такой, что согласно (\ref{eq1})
\begin{equation}\label{eq2D}
M_{\alpha}(\Gamma(C_0^{\,\prime}, f(\gamma_i), D^{\,\prime}))\ge
\delta
>0
\end{equation}
при всех $i\ge i_0\in {\Bbb N}.$

\medskip
Заметим, что ввиду замкнутости отображения $f,$ при достаточно малых
$r>0$
\begin{equation}\label{eq2}
C_0^{\,\prime}\subset f(D)\setminus \overline{f(B(b, r)\cap D)}\,.
\end{equation}
Действительно, если предположить, что включение (\ref{eq2}) не
выполняется при сколь угодно малых $r>0,$ то нашлась бы
последовательность $r_i>0,$ $r_i\rightarrow 0$ при
$i\rightarrow\infty,$ и элементы $y_i\in \overline{f(B(b, r_i)\cap
D)}\cap C_0^{\,\prime}.$ Поскольку $C_0^{\,\prime}$ -- компакт в
$f(D),$ то можно считать, что $y_i\rightarrow y_0\in C_0^{\,\prime}
$ при $i\rightarrow\infty.$ Так как $y_i\in \overline{f(B(b,
r_i)\cap D)}\cap C_0^{\,\prime},$ то при каждом фиксированном $i\in
{\Bbb N}$ найдётся последовательность $y_{ik}\in f(B(b, r_i)\cap D)
$ такая, что $y_{ik}\rightarrow y_i$ при $k\rightarrow\infty.$
Заметим, что $y_{ik}=f(x_{ik}),$ $x_{ik}\in B(b, r_i)\cap D.$

В силу сходимости $y_{1k}$ к $y_1,$ для числа $1/2$ отыщется номер
$k_1$ такой, что $|y_1-y_{1k_1}|<1/2.$ Аналогично, в силу сходимости
$y_{2k}$ к $y_2,$ для числа $1/4$ отыщется номер $k_2$ такой, что
$|y_2-y_{2k_2}|<1/4.$ Вообще, в силу сходимости $y_{mk}$ к $y_m$ для
числа $1/2^m$ отыщется номер $k_m$ такой, что
$|y_m-y_{mk_m}|<1/2^m.$ Но тогда, поскольку по построению
$y_i\rightarrow y_0$ при $i\rightarrow\infty,$ для любого
фиксированного $\varepsilon>0$ будем иметь
$$|y_0-y_{mk_m}|\leqslant |y_0-y_m|+|y_m-y_{mk_m}|\leqslant \varepsilon+1/2^m$$
при всех $m\geqslant M=M(\varepsilon),$ а это означает, что
$y_{mk_m}\rightarrow y_0$ при $m\rightarrow\infty.$ Но, с другой
стороны, $y_{mk_m}=f(x_{mk_m}),$ $x_{mk_m}\in B(b, r_m)\cap D,$
поэтому $y_0\in C(f, b),$ что противоречит замкнутости отображения
$f$ поскольку, с одной стороны, мы имеем, что $C(f,
\partial D)\subset \partial f(D)$ (см. \cite[теорема~3.3]{Vu$_1$}),
а с другой стороны, $y_0\in C(f, \partial D)$ и $y_0\in
C_0^{\,\prime},$ т.е., $y_0$ -- внутренняя точка области $f(D).$
Полученное противоречие указывает на справедливость включения
(\ref{eq2}). Выберем $\delta_0$ настолько малым, чтобы соотношение
(\ref{eq2}) имело место при всех $r\in (2^{\,-i}, \delta_0).$

\medskip
Рассмотрим семейство множеств $\Gamma_i:=\bigcup\limits_{r\in
(2^{\,-i}, \delta_0)}\{\partial f(B(b, r)\cap D)\cap f(D)\}.$
Заметим, что множество $\sigma_r:=\partial f(B(b, r)\cap D)\cap
f(D)$ замкнуто в $f(D).$ Кроме того, заметим, что $\sigma_r$
отделяет $f(\gamma_i)$ от $C_0^{\,\prime}$ в $f(D),$ поскольку
$$f(\gamma_i)\subset f(B(b, r)\cap D):=A,\quad
C_0^{\,\prime}\subset f(D)\setminus \overline{f(B(b, r)\cap
D)}:=B\,,$$
$A$ и $B$ открыты в $f(D)$ и
$$f(D)=A\cup \sigma_r\cup B\,.$$

\medskip
Пусть $\Sigma_i$ -- семейство всех множеств, отделяющих
$f(\gamma_i)$ от $C_0^{\,\prime}$ в $f(D).$ Поскольку $f$ --
открытое замкнутое отображение, мы получим:
\begin{equation}\label{eq7}
(\partial f(B(b, r)\cap D))\cap f(D)\subset f(S(b, r)\cap D), r>0.
\end{equation}
Действительно, пусть $y_0\in (\partial f(B(b, r)\cap D))\cap f(D),$
тогда найдётся последовательность $y_k\in f(B(b, r)\cap D)$ такая,
что $y_k\rightarrow y_0$ при $k\rightarrow \infty,$ где
$y_k=f(x_k),$ $x_k\in B(b, r)\cap D.$ Не ограничивая общности
рассуждений, можно считать, что $x_k\rightarrow x_0$ при
$k\rightarrow\infty.$ Заметим, что случай $x_0\in \partial D$
невозможен, поскольку в этом случае $y_0\in C(f, \partial D),$ что
противоречит замкнутости отображения $f.$ Тогда $x_0\in D.$ Возможны
две ситуации: 1) $x_0\in B(b , r)\cap D$ и 2) $x_0\in S(b , r)\cap
D.$ Заметим, что случай 1) невозможен, поскольку, в этом случае,
$f(x_0)=y_0$ и $y_0$ -- внутренняя точка множества $f(B(b, r)\cap
D),$ что противоречит выбору $y_0.$ Таким образом, включение
(\ref{eq7}) установлено.

\medskip
Здесь и далее объединения вида $\bigcup\limits_{r\in (r_1, r_2)}
\partial f(B(b, r)\cap D)\cap f(D)$ понимаются как семейство множеств. Пусть $\rho^{n-1}\in \widetilde{{\rm
adm}}\bigcup\limits_{r\in (2^{\,-i}, \delta_0)}
\partial f(B(b, r)\cap D)\cap f(D)$ в смысле соотношения (\ref{eq13.4.13}), тогда также
$\rho\in {\rm adm}\bigcup\limits_{r\in (2^{\,-i}, \delta_0)}
\partial f(B(b, r)\cap D)\cap f(D)$ в смысле соотношения (\ref{eq8.2.6}) при
$k=n-1.$ Ввиду (\ref{eq7}) мы получим, что $\rho\in {\rm
adm}\bigcup\limits_{r\in (2^{\,-i}, \delta_0)} f(S(b, r)\cap D)$ и,
следовательно, так как $\widetilde{M}_{q}(\Sigma_i)\geqslant
M_{q(n-1)}(\Sigma_i)$ при произвольном $q\geqslant 1,$ то
$$\widetilde{M}_{p/(n-1)}(\Sigma_i)\geqslant$$
$$\geqslant
\widetilde{M}_{p/(n-1)}\left(\bigcup\limits_{r\in (2^{\,-i},
\delta_0)}
\partial f(B(b, r)\cap D)\cap f(D)\right)\geqslant \widetilde{M}_{p/(n-1)}
\left(\bigcup\limits_{r\in (2^{\,-i}, \delta_0)} f(S(b, r)\cap
D)\right)\geqslant
$$
\begin{equation}\label{eq5}
\ge M_{p}\left(\bigcup\limits_{r\in (2^{\,-i}, \delta_0)} f(S(b,
r)\cap D)\right)\,.
\end{equation}
Однако, ввиду (\ref{eq3}) и (\ref{eq4}),
\begin{equation}\label{eq6}
\widetilde{M}_{p/(n-1)}(\Sigma_i)=\frac{1}{(M_{\alpha}(\Gamma(f(\gamma_i),
C_0^{\,\prime}, f(D))))^{1/(\alpha-1)}}\,.
\end{equation}
По лемме \ref{lemma4}
$$M_{p}\left(\bigcup\limits_{r\in (2^{\,-i}, \delta_0)} f(S(b,
r)\cap D)\right)\geqslant
$$
\begin{equation}\label{eq8}
\geqslant \int\limits_{2^{\,-i}}^{\delta_0}
\frac{dr}{\Vert\,Q\Vert_{s}(r)}= \int\limits_{2^{\,-i}}^{\delta_0}
\frac{dt}{\omega^{\frac{p-n+1}{n-1}}_{n-1}
t^{\frac{n-1}{\alpha-1}}\widetilde{q}_{b}^{\,\frac{1}{\alpha-1}}(t)}\quad\forall\,\,
i\in {\Bbb N}\,, s=\frac{n-1}{p-n+1}\,,\end{equation} где
$\Vert
Q\Vert_{s}(r)=\left(\int\limits_{D(b,r)}Q^{s}(x)\,d{\mathcal{A}}\right)^{\frac{1}{s}}$
-- $L_{s}$-норма функции $Q$ над сферой $S(b,r)\cap D.$
Отметим, что из соотношений (\ref{eq5})--(\ref{eq8}) вытекает, что
интеграл $\int\limits_{2^{\,-i}}^{\delta_0}
\frac{dt}{\omega^{\frac{p-n+1}{n-1}}_{n-1}
t^{\frac{n-1}{\alpha-1}}\widetilde{q}_{b}^{\,\frac{1}{\alpha-1}}(t)}$
сходится, но тогда из условия (\ref{eq9A}) вытекает, что
$\int\limits_{2^{\,-i}}^{\delta_0}
\frac{dt}{\omega^{\frac{p-n+1}{n-1}}_{n-1}
t^{\frac{n-1}{\alpha-1}}\widetilde{q}_{b}^{\,\frac{1}{\alpha-1}}(t)}\rightarrow\infty$
при $i\rightarrow\infty.$

Из соотношений (\ref{eq9A}) и (\ref{eq8}) вытекает, что
$\widetilde{M}_{p/(n-1)}(\Sigma_i)\rightarrow\infty$ при
$i\rightarrow\infty,$ однако, в таком случае, из (\ref{eq6})
следует, что $M_{\alpha}(\Gamma(f(\gamma_i), C_0^{\,\prime},
f(D)))\rightarrow 0$ при $i\rightarrow\infty,$ что противоречит
неравенству (\ref{eq2D}). Полученное противоречие опровергает
предположение о том, что $f$ не имеет непрерывного продолжения в
точку $b\in \partial D.$
\end{proof}~$\Box$

 \medskip
{\it Доказательство теоремы \ref{th1}} вытекает из лемм
\ref{thOS4.1} и \ref{lem1}, а также того факта, что максимальная
кратность $N(f, D)$ замкнутого открытого дискретного отображения $f$
конечна (см., напр., \cite[лемма~3.3]{MS}). $\Box$

\medskip
{\bf 3. Некоторые следствия и замечания.} Ещё один важный результат,
относящийся к устранению особенностей классов Орлича--Соболева,
касается функций конечного среднего колебания (см. \cite{MRSY} и
\cite{IR}).

\medskip
\medskip
В дальнейшем нам понадобится следующее вспомогательное утверждение
(см., напр., \cite[лемма~7.4, гл.~7]{MRSY} и
\cite[лемма~2.2]{RS$_1$}) при $p=n$ и \cite[лемма~2.2]{Sal} при
$p\ne n.$

\medskip
\begin{proposition}\label{pr1A}
{\, Пусть  $x_0 \in {\Bbb R}^n,$ $Q(x)$ -- измеримая по Лебегу
функция, $Q:{\Bbb R}^n\rightarrow [0, \infty],$ $Q\in
L_{loc}^1({\Bbb R}^n).$ Полагаем $A:=A(r_1,r_2,x_0)=\{ x\,\in\,{\Bbb
R}^n : r_1<|x-x_0|<r_2\}$ и
$\eta_0(r)=\frac{1}{Ir^{\frac{n-1}{p-1}}q_{x_0}^{\frac{1}{p-1}}(r)},$
где $I:=I=I(x_0,r_1,r_2)=\int\limits_{r_1}^{r_2}\
\frac{dr}{r^{\frac{n-1}{p-1}}q_{x_0}^{\frac{1}{p-1}}(r)}$ и
$q_{x_0}(r):=\frac{1}{\omega_{n-1}r^{n-1}}\int\limits_{|x-x_0|=r}Q(x)\,d{\mathcal
H}^{n-1}$ -- среднее интегральное значение функции $Q$ над сферой
$S(x_0, r).$ Тогда
\begin{equation}\label{eq10A}
\frac{\omega_{n-1}}{I^{p-1}}=\int\limits_{A} Q(x)\cdot
\eta_0^p(|x-x_0|)\ dm(x)\leqslant\int\limits_{A} Q(x)\cdot
\eta^p(|x-x_0|)\ dm(x)
\end{equation}
для любой измеримой по Лебегу функции $\eta :(r_1,r_2)\rightarrow
[0,\infty]$ такой, что
$\int\limits_{r_1}^{r_2}\eta(r)dr=1. $ }
\end{proposition}

\medskip Будем говорить, что локально интегрируемая
функция ${\varphi}:D\rightarrow{\Bbb R}$ имеет {\it конечное среднее
колебание} в точке $x_0\in D$, пишем $\varphi\in FMO(x_0),$ если
%
%
%
%
$\limsup\limits_{\varepsilon\rightarrow
0}\frac{1}{\Omega_n\varepsilon^n}\int\limits_{B( x_0,\,\varepsilon)}
|{\varphi}(x)-\overline{{\varphi}}_{\varepsilon}|\, dm(x)<\infty,$
%
%
где
$\overline{{\varphi}}_{\varepsilon}=\frac{1}
{\Omega_n\varepsilon^n}\int\limits_{B(x_0,\,\varepsilon)}
{\varphi}(x)\, dm(x).$
\medskip
Заметим, что, как известно, $\Omega_n\varepsilon^n=m(B(x_0,
\varepsilon)).$ Имеет место следующая

\medskip
\begin{theorem}\label{th3}
{\sl\, Пусть область $D$ локально связна в каждой точке своей
границы, $n\ge 3,$ отображение $f:D\rightarrow {\Bbb R}^n$ класса
$W_{loc}^{1, \varphi}(D)$ с конечным искажением является
ограниченным, открытым, дискретным и замкнутым, а граница области
$D^{\,\prime}=f(D)$ является сильно достижимой относительно
$\alpha$-модуля, $n-1<\alpha\leqslant n.$ Тогда $f$ имеет
непрерывное продолжение в точку $x_0\in \partial D,$ если выполнено
условие (\ref{eqOS3.0a}) и, кроме того, найдётся функция $Q\in
L_{loc}^1(D),$ такая что $K_{I,\alpha}(x, f)\leqslant Q(x)$ при
почти всех $x\in D$ и $Q\in FMO (x_0).$ }
\end{theorem}

\medskip
\begin{proof}
Достаточно показать, что условие $Q\in FMO(x_0)$ влечёт расходимость
интеграла (\ref{eq9}), поскольку в этому случае необходимое
заключение будет следовать из теоремы \ref{th1}. Заметим, что для
функций класса $FMO$  в точке $x_0$
\begin{equation}\label{eq31*}
\int\limits_{\varepsilon<|x|<{e_0}}\frac{Q(x+x_0)\, dm(x)}
{\left(|x| \log \frac{1}{|x|}\right)^n} = O \left(\log\log
\frac{1}{\varepsilon}\right)
\end{equation}
при  $\varepsilon \rightarrow 0 $ и для некоторого $e_0>0,$ $e_0
\leqslant {\rm dist}\,\left(0,\partial D\right).$ При
$\varepsilon_0<r_0:={\rm dist}\,\left(0,\partial D\right)$ полагаем
$\psi(t):=\frac{1}{\left(t\,\log{\frac1t}\right)^{n/{\alpha}}},$
$I(\varepsilon,
\varepsilon_0):=\int\limits_{\varepsilon}^{\varepsilon_0}\psi(t) dt
\geqslant \log{\frac{\log{\frac{1}
{\varepsilon}}}{\log{\frac{1}{\varepsilon_0}}}}$ и
$\eta(t):=\psi(t)/I(\varepsilon, \varepsilon_0).$ Заметим, что
$\int\limits_{\varepsilon}^{\varepsilon_0}\eta(t)dt=1,$ кроме того,
из соотношения (\ref{eq31*}) вытекает, что
\begin{equation}\label{eq32*}
\frac{1}{I^{\alpha}(\varepsilon,
\varepsilon_0)}\int\limits_{\varepsilon<|x|<\varepsilon_0}
Q(x+x_0)\cdot\psi^{\alpha}(|x|)
 \ dm(x)\leqslant C\left(\log\log\frac{1}{\varepsilon}\right)^{1-{\alpha}}\rightarrow
 0
 \end{equation}
при $\varepsilon\rightarrow 0.$ Из соотношений (\ref{eq10A}) и
(\ref{eq32*}) вытекает, что интеграл вида (\ref{eq9}) расходится,
что и требовалось установить.
\end{proof} $\Box$

\medskip
Из леммы \ref{lem1} и замечания \ref{rem1} вытекает следующая

\medskip
\begin{theorem}\label{th2}
{\sl\, Пусть область $D$ локально связна в каждой точке своей
границы, отображение $f:D\rightarrow {\Bbb C}$ класса $W_{loc}^{1,
1}(D)$ с конечным искажением является ограниченным, открытым,
дискретным и замкнутым, а граница области $D^{\,\prime}=f(D)$
является сильно достижимой относительно $\alpha$-модуля,
$1<\alpha\leqslant 2.$ Тогда $f$ имеет непрерывное продолжение в
точку $x_0\in \partial D,$ если найдётся измеримая по Лебегу функция
$Q:D\rightarrow [0, \infty],$ такая что $K_{I,\alpha}(x, f)\leqslant
Q(x)$ при почти всех $x\in D$ и при некотором $\varepsilon_0>0,$
либо выполнено условие расходимости интеграла
%
$$\int\limits_{0}^{\varepsilon_0}
\frac{dt}{t^{\frac{1}{\alpha-1}}q_{x_0}^{\,\frac{1}{\alpha-1}}(t)}=\infty\,,$$
%
где $q_{x_0}(r):=\frac{1}{2\pi
}\int\limits_0^{2\pi}Q(x_0+re^{i\theta})\,d\theta,$ либо $Q\in
FMO(x_0).$}
\end{theorem}

{\bf 4. Связь нижних и кольцевых $Q$-отображений в граничных
точках.} Справедливо следующее утверждение.

\medskip
\begin{theorem}\label{th4}
{\sl\, Пусть $x_0\in \partial D,$ отображение $f:D\rightarrow {\Bbb
R}^n$ является ограниченным, открытым, дискретным, и замкнутым
нижним $Q$-отображением относительно $p$-модуля в области
$D\subset{\Bbb R}^n,$ $Q\in L_{loc}^1,$ $n-1<p\leqslant n,$ и
$\alpha:=\frac{p}{p-n+1}.$ Тогда для каждого
$\varepsilon_0<d_0:=\sup\limits_{x\in D}|x-x_0|$ и каждого компакта
$C_2\subset D\setminus B(x_0, \varepsilon_0)$ найдётся
$\varepsilon_1,$ $0<\varepsilon_1<\varepsilon_0,$ такое, что для
каждого $\varepsilon\in (0, \varepsilon_1)$ и каждого компакта
$C_1\subset \overline{B(x_0, \varepsilon)}\cap D$ выполнено
неравенство
\begin{equation}\label{eq3A}M_{\alpha}(f(\Gamma(C_1, C_2, D)))\leqslant \int\limits_{A(x_0,
\varepsilon, \varepsilon_1)}Q^{\frac{n-1}{p-n+1}}(x)
\eta^{\alpha}(|x-x_0|)dm(x)\,,
\end{equation}
где $A(x_0, \varepsilon, \varepsilon_1)=\{x\in {\Bbb R}^n:
\varepsilon<|x-x_0|<\varepsilon_1\}$ и $\eta: (\varepsilon,
\varepsilon_1)\rightarrow [0,\infty]$ -- произвольная измеримая по
Лебегу функция такая, что
\begin{equation}\label{eq6B}
\int\limits_{\varepsilon}^{\varepsilon_1}\eta(r)dr=1\,.
\end{equation}
 }
\end{theorem}

\begin{proof} Зафиксируем $\varepsilon_0$ такое, как в условии
теоремы. Пусть компакт $C_2$ удовлетворяет условию $C_2\subset
D\setminus B(x_0, \varepsilon_0).$ Тогда $K:=f(C_2)$ -- компакт в
$f(D).$ Пусть $K^{\,*}$ -- полный прообраз компакта $K$ при
отображении $f$ в $D.$ Поскольку по условию отображение $f$
замкнутое, множество $K^{\,*}$ является компактом в $D$ (см.
\cite[теорема~3.3]{Vu$_1$}), поэтому найдётся $\varepsilon_1$ такое,
что $K^{\,*}\cap \overline{B(x_0, \varepsilon_1)}=\varnothing.$

\medskip
Далее покажем, что
\begin{equation}\label{eq8A}
K\subset D^{\,\prime}\setminus \overline{f(B(x_0, r)\cap D)}\,, r\in
(0, \varepsilon_1)\,.
\end{equation}
Предположим противное, а именно, что найдётся $\zeta_0\in K\cap
\overline{f(B(x_0, r)\cap D)},$ $r\in (0, \varepsilon_1).$ Тогда
$\zeta_0=\lim\limits_{k\rightarrow\infty} \zeta_k,$ где $\zeta_k\in
f(B(x_0, r)\cap D).$ Отсюда $\zeta_k=f(\xi_k),$ $\xi_k\in B(x_0,
r)\cap D.$ Так как $\overline{B(x_0, r)\cap D}$ -- компакт, то из
последовательности $\xi_k$ можно выделить сходящуюся
подпоследовательность $\xi_{k_l}\rightarrow \xi_0\in
\overline{B(x_0, r)\cap D}.$ Случай $\xi_0\in \partial D$
невозможен, поскольку $f$ -- замкнутое отображение и, значит,
сохраняет границу: $C(f,
\partial D)\subset \partial f(D),$ но у нас $\zeta_0$ -- внутренняя
точка $D^{\,\prime}.$ Пусть $\xi_0$ -- внутренняя точка $D.$ По
непрерывности отображения $f$ имеем $f(\xi_0)=\zeta_0.$ Но тогда
одновременно $\xi_0\in B(x_0, \varepsilon_1)\cap D$ и $\xi_0\in
f^{\,-1}(K),$ что противоречит выбору $\varepsilon_1.$ Таким
образом, $K\cap \overline{f(B(x_0, r)\cap D)}=\varnothing$ при $r\in
(0, \varepsilon_1)$ и, значит, имеет место включение (\ref{eq8A}).

\medskip
Зафиксируем $\varepsilon\in (0, \varepsilon_1)$ и компакт
$C_1\subset \overline{B(x_0, \varepsilon)}\cap D.$ Рассмотрим
семейство множеств $\Gamma_{\varepsilon}:=\bigcup\limits_{r\in
(\varepsilon, \varepsilon_1)}\{\partial f(B(x_0, r)\cap D)\cap
f(D)\}.$ Заметим, что множество $\sigma_r:=\partial f(B(x_0, r)\cap
D)\cap f(D)$ замкнуто в $f(D).$ Кроме того, заметим, что $\sigma_r$
при $r\in (\varepsilon, \varepsilon_1)$ отделяет $f(C_1)$ от
$K=f(C_2)$ в $f(D),$ поскольку
$$f(C_1)\subset f(B(x_0, r)\cap D):=A,\quad
f(C_2)\subset f(D)\setminus \overline{f(B(x_0, r)\cap D)}:=B\,,$$
$A$ и $B$ открыты в $f(D)$ и
$$f(D)=A\cup \sigma_r\cup B\,.$$

\medskip
Пусть $\Sigma_{\varepsilon}$ -- семейство всех множеств, отделяющих
$f(C_1)$ от $f(C_2)$ в $f(D).$ Поскольку $f$ -- открытое замкнутое
отображение, мы получим, что
\begin{equation}\label{eq7A}
(\partial f(B(x_0, r)\cap D))\cap f(D)\subset f(S(x_0, r)\cap D),
r>0.
\end{equation}
Действительно, пусть $\zeta_0\in (\partial f(B(x_0, r)\cap D))\cap
f(D),$ тогда найдётся последовательность $\zeta_k\in f(B(x_0, r)\cap
D)$ такая, что $\zeta_k\rightarrow \zeta_0$ при $k\rightarrow
\infty,$ где $\zeta_k=f(\xi_k),$ $\xi_k\in B(x_0, r)\cap D.$ Не
ограничивая общности рассуждений, можно считать, что
$\xi_k\rightarrow \xi_0$ при $k\rightarrow\infty.$ Заметим, что
случай $\xi_0\in \partial D$ невозможен, поскольку в этом случае
$\zeta_0\in C(f, \partial D),$ что противоречит замкнутости
отображения $f.$ Тогда $\xi_0\in D.$ Возможны две ситуации: 1)
$\xi_0\in B(x_0 , r)\cap D$ и 2) $\xi_0\in S(x_0 , r)\cap D.$
Заметим, что случай 1) невозможен, поскольку, в этом случае,
$f(\xi_0)=\zeta_0$ и $\zeta_0$ -- внутренняя точка множества
$f(B(x_0, r)\cap D),$ что противоречит выбору $\zeta_0.$ Таким
образом, включение (\ref{eq7A}) установлено.

\medskip
Здесь и далее объединения вида $\bigcup\limits_{r\in (r_1, r_2)}
\partial f(B(x_0, r)\cap D)\cap f(D)$ понимаются как семейство множеств. Пусть $\rho^{n-1}\in \widetilde{{\rm
adm}}\bigcup\limits_{r\in (\varepsilon, \varepsilon_1)}
\partial f(B(x_0, r)\cap D)\cap f(D)$ в смысле соотношения (\ref{eq13.4.13}), тогда также
$\rho\in {\rm adm}\bigcup\limits_{r\in (\varepsilon, \varepsilon_1)}
\partial f(B(x_0, r)\cap D)\cap f(D)$ в смысле соотношения (\ref{eq8.2.6}) при
$k=n-1.$ Ввиду (\ref{eq7A}) мы получим, что $\rho\in {\rm
adm}\bigcup\limits_{r\in (\varepsilon, \varepsilon_1)} f(S(x_0,
r)\cap D)$ и, следовательно, так как
$\widetilde{M}_{q}(\Sigma_{\varepsilon})\geqslant
M_{q(n-1)}(\Sigma_{\varepsilon})$ при произвольном $q\geqslant 1,$
то
$$\widetilde{M}_{p/(n-1)}(\Sigma_{\varepsilon})\geqslant$$
$$\geqslant
\widetilde{M}_{p/(n-1)}\left(\bigcup\limits_{r\in (\varepsilon,
\varepsilon_1)}
\partial f(B(x_0, r)\cap D)\cap f(D)\right)\geqslant \widetilde{M}_{p/(n-1)}\left(\bigcup
\limits_{r\in (\varepsilon, \varepsilon_1)} f(S(x_0, r)\cap
D)\right)\geqslant
$$
\begin{equation}\label{eq5A}
\ge M_{p}\left(\bigcup\limits_{r\in (\varepsilon, \varepsilon_1)}
f(S(x_0, r)\cap D)\right)\,.
\end{equation}
Однако, ввиду (\ref{eq3}) и (\ref{eq4}),
\begin{equation}\label{eq6A}
\widetilde{M}_{p/(n-1)}(\Sigma_{\varepsilon})=\frac{1}{(M_{\alpha}(\Gamma(f(C_1),
f(C_2), f(D))))^{1/(\alpha-1)}}\,.
\end{equation}
По лемме \ref{lemma4}
$$M_{p}\left(\bigcup\limits_{r\in (\varepsilon, \varepsilon_1)} f(S(x_0,
r)\cap D)\right)\geqslant
$$
\begin{equation}\label{eq8B}
\geqslant \int\limits_{\varepsilon}^{\varepsilon_1}
\frac{dr}{\Vert\,Q\Vert_{s}(r)}=
\int\limits_{\varepsilon}^{\varepsilon_1}
\frac{dt}{\omega^{\frac{p-n+1}{n-1}}_{n-1}
t^{\frac{n-1}{\alpha-1}}\widetilde{q}_{x_0}^{\,\frac{1}{\alpha-1}}(t)}\quad\forall\,\,
i\in {\Bbb N}\,, s=\frac{n-1}{p-n+1}\,,\end{equation} где
$\Vert
Q\Vert_{s}(r)=\left(\int\limits_{D(x_0,r)}Q^{s}(x)\,d{\mathcal{A}}\right)^{\frac{1}{s}}$
-- $L_{s}$-норма функции $Q$ над сферой $S(x_0,r)\cap D.$ Тогда из
(\ref{eq5A})--(\ref{eq8B}) вытекает, что
\begin{equation}\label{eq9C}
M_{\alpha}(\Gamma(f(C_1), f(C_2), f(D)))\leqslant
\frac{\omega_{n-1}}{I^{\alpha-1}}\,,
\end{equation}
где $I=\int\limits_{\varepsilon}^{\varepsilon_1}\
\frac{dr}{r^{\frac{n-1}{\alpha-1}}q_{x_0}^{\frac{1}{\alpha-1}}(r)}.$
Заметим, что $f(\Gamma(C_1,C_2, D))\subset \Gamma(f(C_1), f(C_2),
f(D)),$ так что из (\ref{eq9C}) вытекает, что
\begin{equation}\label{eq10D}
M_{\alpha}(f(\Gamma(C_1,C_2, D)))\leqslant
\frac{\omega_{n-1}}{I^{\alpha-1}}\,.
\end{equation}
Завершает доказательство применение предложения \ref{pr1A}.~$\Box$
\end{proof}

\medskip
{\bf 5. Устранение изолированных особенностей классов
Орлича--Соболева.} В работе \cite{Sev$_1$} нами было показано, что
открытые дискретные замкнутые ограниченные отображения области
$D\setminus\{x_0\}$ классов $W_{loc}^{1, \varphi},$ удовлетворяющие
условиям вида (\ref{eqOS3.0a})-(\ref{eq9}) при $p=n$, продолжаются
по непрерывности в точку $x_0\in D,$ как только предельные множества
$C(f, x_0)$ и $C(f, \partial D)$ не пересекаются. Ниже этот
результат будет показан для случая, когда $p\in (n-1, n]$ и условие
$C(f, x_0)\cap C(f,
\partial D)=\varnothing,$ вообще говоря, может нарушаться.
Аналог следующего утверждения доказан, например, в
\cite[лемма~2.2]{GSS}.

\medskip
\begin{lemma}\label{lem3.1!} {\sl\, Пусть $x_0\in D,$
$f:D\setminus\{x_0\}\rightarrow {\Bbb R}^n,$ $n\geqslant 2,$
является открытым дискретным и замкнутым ограниченным
$Q$-отображением относительно $p$-модуля в точке $x_0\in D.$
Предположим, что найдётся $\varepsilon_0>0$ и измеримая по Лебегу
функция $\psi(t):(0, \varepsilon_0)\rightarrow [0,\infty]$ такие,
что при всех $\varepsilon\in (0, \varepsilon_0)$
\begin{equation} \label{eq5C}
0<I(\varepsilon,
\varepsilon_0):=\int\limits_{\varepsilon}^{\varepsilon_0}\psi(t)dt <
\infty\,, I(\varepsilon, \varepsilon_0)\rightarrow
\infty\quad\text{при}\quad\varepsilon\rightarrow 0
\end{equation}
и, кроме того, при $\varepsilon\rightarrow 0$
\begin{equation} \label{eq4A}
\int\limits_{\varepsilon<|x-x_0|<\varepsilon_0}Q^{\,s}(x)\cdot\psi^{\,\alpha}(|x-x_0|)
\ dm(x)\,=\,o\left(I^{\,\alpha}(\varepsilon,
\varepsilon_0)\right)\,, \alpha=\frac{p}{p-n+1}\,,
s=\frac{n-1}{p-n+1}\,.
\end{equation}

Если $\Gamma$ -- семейство всех открытых кривых
$\gamma(t):(0,1)\rightarrow D\setminus\{x_0\}$ таких, что
$\gamma(t_k)\rightarrow x_0$ при некоторой последовательности
$t_k\rightarrow 0,$ $\gamma(t)\not\equiv x_0,$ то
$M_{\alpha}\left(f(\Gamma)\right)=0.$}
\end{lemma}

\medskip
\begin{proof}
Не ограничивая общности рассуждений, можно считать, что $x_0=0.$
Заметим, что
\begin{equation}\label{eq12*}
\Gamma > \bigcup\limits_{i=1}^\infty\,\, \Gamma_i\,,
\end{equation}
где $\Gamma_i$ -- семейство кривых $\alpha_i(t):(0,1)\rightarrow
{{\Bbb R}^n}$ таких, что $\alpha_i(1)\in
\{0<|x|=r_i<\varepsilon_0\},$ где $r_i$ -- некоторая
последовательность, такая что $r_i\rightarrow 0$ при $i\rightarrow
\infty$ и $\alpha_i(t_k)\rightarrow 0$ при $k\rightarrow\infty$ для
той же самой последовательности $t_k\rightarrow 0$ при
$k\rightarrow\infty.$ Зафиксируем $i\ge 1.$ Можно считать, что
$r_i<d_0:=\sup\limits_{x\in D\setminus\{0\}}|x|.$ Положим $C_2:=S(0,
r_i),$ заметим, что $C_2\subset D\setminus (B(0, r_i)\cup\{0\}).$
Тогда согласно теореме \ref{th4} найдётся $k_i\in (0, r_i)$ такое,
что для любого $\varepsilon\in (0, k_i)$ и любого компакта
$C_1\subset \overline{B(0, \varepsilon)}\setminus\{0\}$ выполнено
соотношение
\begin{equation}\label{eq4B}
M_{\alpha}(f(\Gamma(C_1, C_2, D\setminus\{0\})))\leqslant
\int\limits_{A(0, \varepsilon, k_i)}Q^{\frac{n-1}{p-n+1}}(x)
\eta^{\alpha}(|x|)dm(x)\,,
\end{equation}
где $\eta: (\varepsilon, k_i)\rightarrow [0,\infty]$ -- произвольная
измеримая по Лебегу функция такая, что
\begin{equation}\label{eq5D}
\int\limits_{\varepsilon}^{k_i}\eta(r)dr=1\,.
\end{equation}

Поскольку ввиду (\ref{eq5C}) мы имеем условие $I(\varepsilon,
\varepsilon_0)\rightarrow \infty$ при $\varepsilon\rightarrow 0,$
найдётся $l_i\in (0, k_i]$ такое, что $I(\varepsilon, k_i)>0$ при
всех $\varepsilon\in(0, l_i).$
Заметим, что при указанных $\varepsilon>0$ функция
$$\eta(t)=\left\{
\begin{array}{rr}
\psi(t)/I(\varepsilon, k_i), &   t\in (\varepsilon,
k_i),\\
0,  &  t\in {\Bbb R}\setminus (\varepsilon, k_i)
\end{array}
\right. $$ удовлетворяет условию нормировки вида (\ref{eq5D}) в
кольце
$A(0, \varepsilon, k_i)=\{x\in {\Bbb R}^n:\varepsilon<|x|< k_i \}.$
Положим $C_2=S(0, \varepsilon),$ тогда ввиду (\ref{eq4B}) мы
получим, что
$$M_{\alpha}\left(f\left(\Gamma\left(S(0, \varepsilon),\,S(0,
r_i),\,D\setminus\{0\})\right)\right)\right)\le$$
\begin{equation}\label{eq11*}
\le \int\limits_{A(0, \varepsilon, \varepsilon_0)} Q(x)\cdot
\eta^{\,\alpha}(|x|)\ dm(x)\,\leqslant {\frak F}_i(\varepsilon),
\end{equation}
где
${\frak F}_i(\varepsilon)=\,\frac{1}{\left(I(\varepsilon,
k_i)\right)^{\,\alpha}}\int\limits_{\varepsilon<|x|<\varepsilon_0}\,Q(x)\,\psi^{\,\alpha}(|x|)\,dm(x).$
Принимая во внимание (\ref{eq4A}), получим, что ${\frak
F}_i(\varepsilon)\rightarrow 0$ при $\varepsilon\rightarrow 0.$
Заметим, что при каждом $\varepsilon\in (0, k_i)$
\begin{equation}\label{eq5*C}
\Gamma_i>\Gamma\left(S(0, \varepsilon),\,S(0, r_i),
D\setminus\{0\}\right)\,.
\end{equation}
Таким образом, при каждом фиксированном $i=1,2,\ldots$  из
(\ref{eq11*}) и (\ref{eq5*C}) получаем, что
\begin{equation}\label{eq6*}
M_{\alpha}(f(\Gamma_i))\leqslant {\frak F}_i(\varepsilon)\rightarrow
0
\end{equation}
при $\varepsilon\rightarrow 0$ и каждом фиксированном $i\in {\Bbb
N}.$ Однако, левая часть неравенства (\ref{eq6*}) не зависит от
$\varepsilon$ и, следовательно, $M_{\alpha}(f(\Gamma_i))=0.$
Наконец, из (\ref{eq12*}) и свойства полуаддитивности
$\alpha$-модуля вытекает, что $M_{\alpha}(f(\Gamma))=0.$
\end{proof}

\medskip
Справедливо следующее утверждение.

\medskip
\begin{lemma}\label{lem2}
{\sl\, Пусть $x_0\in D,$ отображение $f:D\setminus\{x_0\}\rightarrow
{\Bbb R}^n$ является ограниченным, открытым, дискретным, и замкнутым
нижним $Q$-отображением относительно $p$-модуля, $n-1<p\leqslant n.$
Тогда отображение $f:D\setminus\{x_0\}\rightarrow {\Bbb R}^n$ имеет
непрерывное продолжение в точку $x_0,$ если при некотором
$\varepsilon_0<{\rm dist}\,(x_0, \partial D)$ и всех $\varepsilon\in
(0, \varepsilon_0)$
\begin{equation} \label{eq5B}
0<I(\varepsilon,
\varepsilon_0):=\int\limits_{\varepsilon}^{\varepsilon_0}\psi(t)dt <
\infty\,,\,\,\, I(\varepsilon, \varepsilon_0)\rightarrow
\infty\quad\text{при}\quad\varepsilon\rightarrow 0\,,
\end{equation}
и, кроме того, при $\varepsilon\rightarrow 0$
\begin{equation} \label{eq4*}
\int\limits_{\varepsilon<|x-x_0|<\varepsilon_0}Q^{\,s}(x)\cdot\psi^{\,\alpha}(|x-x_0|)
\ dm(x)\,=\,o\left(I^{\,\alpha}(\varepsilon,
\varepsilon_0)\right)\,, \alpha=\frac{p}{p-n+1}\,,
s=\frac{n-1}{p-n+1}\,.
\end{equation}
}
\end{lemma}

\medskip
\begin{proof} Не ограничивая общности рассуждений, можно считать,
что $x_0=0.$ Предположим противное, а именно, что отображение $f$ не
может быть продолжено по непрерывности в точку $x_0=0.$ Тогда
найдутся две последовательности $x_j$ и $x_j^{\,\prime},$
принадлежащие $D\setminus\left\{0\right\},$ $x_j\rightarrow 0,\quad
x_j^{\,\prime}\rightarrow 0,$ такие, что
$|f(x_j)-f(x_j^{\,\prime})|\geqslant a>0$ для всех $j\in {\Bbb N}.$
Не ограничивая общности рассуждений, можно считать, что $x_j$ и
$x_j^{\,\prime}$ лежат внутри шара $B(0, \varepsilon_0),$ где
$\varepsilon_0$ -- из условия леммы. Полагаем
$r_j=\max{\left\{|x_j|,\,|x_j^{\,\prime}|\right\}}.$ Соединим точки
$x_j$ и $x_j^{\,\prime}$ замкнутой кривой, лежащей в $\overline{B(0,
r_j)}\setminus\left\{0\right\}.$ Обозначим эту кривую символом $C_j$
и рассмотрим конденсатор
$E_j=\left(D\setminus\left\{0\right\}\,,C_j\right)$ (понятие
конденсатора и $p$-ёмкости конденсатора см., напр., в \cite{GSS}).

В силу открытости и непрерывности отображения $f,$ пара $f(E_j)$
также является конденсатором. Для произвольного конденсатора
$E=(A,\,C)$ в ${\Bbb R}^n$ через $\Gamma_E$ будет обозначаться
семейство всех кривых вида $\gamma:[a,\,b)\rightarrow A$ таких, что
$\gamma(a)\in C$ и $|\gamma|\cap\left(A\setminus
F\right)\ne\varnothing$ для произвольного компакта $F\subset A.$
Рассмотрим семейства кривых $\Gamma_{E_j}$ и $\Gamma_{f(E_j)}.$
Пусть $\Gamma_j^{\,*}$ -- семейство всех максимальных под\-ня\-тий
семейства кривых $\Gamma_{f(E_j)}$ при отображении $f$ с началом в
$C_j,$ лежащих в $D\setminus\left\{0\right\}.$ Заметим, что
$\Gamma_j^{\,*}\subset \Gamma_{E_j}.$ Поскольку
$\Gamma_{f(E_j)}>f(\Gamma_j^{\,*}),$ мы получим:
\begin{equation}\label{eq32*!}
M_{\alpha}\left(\Gamma_{f(E_j)}\right)\leqslant
M_{\alpha}\left(f(\Gamma_j^{*})\right)\leqslant
M_{\alpha}\left(f(\Gamma_{E_j})\right)\,.
\end{equation}
Заметим, что семейство $\Gamma_{E_j}$ может быть разбито на два
подсемейства:
\begin{equation}\label{eq33*!}
\Gamma_{E_j}\,=\,\Gamma_{E_{j_1}}\cup \Gamma_{E_{j_2}}\,,
\end{equation}
где $\Gamma_{E_{j_1}}$ -- семейство всех кривых
$\alpha(t):[a,\,c)\rightarrow D\setminus\left\{0\right\}$ с началом
в $C_j,$ таких что найдётся $t_k\in [a,\,c):$
$\alpha(t_k)\rightarrow 0$ при $t_k\rightarrow c-0;$
$\Gamma_{E_{j_2}}$ -- семейство всех кривых
$\alpha(t):[a,\,c)\rightarrow D\setminus\left\{0\right\}$ с началом
в $C_j,$ таких что найдётся $t_k\in [a,\,c):$ ${\rm
dist}\left(\alpha(t_k),\partial D\right)\rightarrow 0$ при
$t_k\rightarrow c-0.$

В силу соотношений (\ref{eq32*!}) и (\ref{eq33*!}),
\begin{equation}\label{eq34*!}
M_{\alpha}\left(\Gamma_{f(E_j)}\right)\leqslant
M_{\alpha}(f(\Gamma_{E_{j_1}}))\,+\,M_{\alpha}(f(\Gamma_{E_{j_2}}))\,.
\end{equation}

По лемме \ref{lem3.1!} $M_{\alpha}(f(\Gamma_{E_{j_1}}))=0.$ Кроме
того, заметим, что $\Gamma_{E_{j_2}}>\Gamma(S(0, r_j), S(0,
\varepsilon_0), D\setminus\{0\}).$ Для континуума $C_2:=S(0,
\varepsilon_0)$ имеем, что $C_2\subset D\setminus (B(0,
\varepsilon)\cup\{0\}).$ Выберем $\varepsilon_1$ таким, что для
любого $\varepsilon\in (0, \varepsilon_1)$ и любого континуума
$C_1\subset \overline{B(0, \varepsilon)}$ имели бы место условия
(\ref{eq3A})--(\ref{eq6B}) (что возможно в силу теоремы \ref{th4}).
Ввиду второго условия в (\ref{eq5B}) мы получим, что $I(\varepsilon,
\varepsilon_1)>0$ при всех $\varepsilon\in (0, \varepsilon_2)$ и
некотором $\varepsilon_2\in (0, \varepsilon_1).$ Так как
$r_j\rightarrow 0$ при $j\rightarrow\infty,$ найдётся $n_0\in {\Bbb
N}$ такой, что $r_j<\varepsilon_2.$ При $j\geqslant n_0$ рассмотрим
семейство функций
$$ \eta_{j}(t)= \left\{
\begin{array}{rr}
\psi(t)/I(r_j, \varepsilon_1), &   t\in (r_j,\, \varepsilon_1),\\
0,  &  t\in {\Bbb R} \setminus (r_j,\, \varepsilon_1)
\end{array}
\right. .$$
Имеем:
$\int\limits_{r_j}^{\varepsilon_1}\,\eta_{j}(t) dt
=\,\frac{1}{I\left(r_j, \varepsilon_1\right)}\int\limits_{r_j}
^{\varepsilon_1}\,\psi(t)dt=1.$ Полагая $C_1:=S(0, r_j)$ из условий
(\ref{eq3A})--(\ref{eq6B}) и учитывая (\ref{eq34*!}), получаем, что
$$M_{\alpha}(f(\Gamma_{E_{j}}))\leqslant\,\frac{1}{{I(r_j,
\varepsilon_1)}^{\alpha}}\int\limits_{r_j<|x|<\varepsilon_1}\,Q(x)\,\psi^{\,\alpha}(|x|)\,dm(x)\,,
$$
откуда, переходя к пределу при $m\rightarrow\infty,$ получим
соотношение
$$
M_{\alpha}(f(\Gamma_{E_{j}}))\leqslant\, {\cal
S}(r_j):=\frac{1}{{I(r_j,
\varepsilon_1)}^n}\int\limits_{r_j<|x|<\varepsilon_1}\,Q(x)\,\psi^{\,\alpha}(\vert
x\vert)\,dm(x).
$$
%
В силу условия (\ref{eq4*}),  ${\cal S}(r_j)\,\rightarrow\, 0$ при
$j\rightarrow \infty.$ Окончательно, по \cite[предложение~10.2,
гл.~II]{Ri},
\begin{equation}\label{eq8C}
{\rm cap}_{\alpha}\,f(E_j)\,\rightarrow\,0\,, j\rightarrow\infty\,.
\end{equation}
С другой стороны, поскольку отображение $f$ ограничено, множество
$\overline{{\Bbb R}^n}\setminus f(D\setminus\{0\})$ имеет
положительную ёмкость, так что ввиду \cite[лемма~2.1]{GSS}
\begin{equation}\label{eq7B}
{\rm cap}_{\alpha}\,(f(D\setminus\{0\}), f(C_j))\geqslant\delta>0
\end{equation}
при всех $j\in {\Bbb N}$ и некотором $\delta>0.$ Соотношения
(\ref{eq8C}) и (\ref{eq7B}) противоречат друг другу, что опровергает
исходное предположение.~$\Box$
\end{proof}

\medskip
Из лемм \ref{lemma4} и \ref{thOS4.1} и теоремы \ref{th4} вытекает
следующее утверждение.

\begin{theorem}\label{th5}
{\sl\, Пусть $n\ge 3,$ $x_0\in D,$ и пусть отображение
$f:D\setminus\{x_0\}\rightarrow {\Bbb R}^n$ класса $W_{loc}^{1,
\varphi}(D)$ с конечным искажением является ограниченным, открытым,
дискретным и замкнутым. Предположим, $n-1<\alpha\leqslant n.$ Тогда
$f$ имеет непрерывное продолжение в точку $x_0,$ если имеет место
условие (\ref{eqOS3.0a})
и, кроме того, найдётся измеримая по Лебегу функция $Q:D\rightarrow
[0, \infty],$ такая что $K_{I,\alpha}(x, f)\leqslant Q(x)$ при почти
всех $x\in D,$ некотором $\varepsilon_0>0$ и всех $\varepsilon\in
(0, \varepsilon_0)$ выполнены условия (\ref{eq9}). Аналогичное
заключение имеет место, если $Q\in FMO(0).$}
\end{theorem}

\medskip
КОНТАКТНАЯ ИНФОРМАЦИЯ

\medskip
\noindent{{\bf Евгений Александрович Севостьянов} \\
Житомирский государственный университет им.\ И.~Франко\\
ул. Большая Бердичевская, 40 \\
г.~Житомир, Украина, 10 008 \\ тел. +38 066 959 50 34 (моб.),
e-mail: esevostyanov2009@mail.ru}

\end{document}